%% file: HFL_BraidConj_v6.tex
\documentclass[10pt]{amsart}
\theoremstyle{plain}
\usepackage{amsmath, amsfonts, amsthm, amssymb,amscd,bbold}
\usepackage{graphics, hyperref, enumerate}
\usepackage{color}
\usepackage{epsfig, pinlabel}
\usepackage[all]{xy}
\graphicspath{ {Figures/} }

\title[$SKh$, $HFL$ and the braid group]{Categorified invariants and the braid group}
\author{John A. Baldwin}
\thanks{JAB was partially supported by NSF grant number DMS-1104688. }
\address{Boston College; Department of Mathematics; 301 Carney Hall; Chestnut Hill, MA 02467}
\email{john.baldwin@bc.edu}

\author{J. Elisenda Grigsby}
\thanks{JEG was partially supported by NSF grant number DMS-0905848 and NSF CAREER award DMS-1151671.}
\address{Boston College; Department of Mathematics; 301 Carney Hall; Chestnut Hill, MA 02467}
\email{julia.grigsby@bc.edu}

\theoremstyle{plain}
\newtheorem{theorem}{Theorem}
\newtheorem{lemma}{Lemma}[section]
\newtheorem{proposition}{Proposition}[section]
\newtheorem{corollary}{Corollary}

\theoremstyle{definition}

\newtheorem{definition}{Definition}[section]
\newtheorem{remark}{Remark}[section]
\newtheorem{question}{Question}[section]

\newcommand{\Szabo}{{Szab{\'o}} }

\newcommand{\R}{\ensuremath{\mathbb{R}}}
\newcommand{\Z}{\ensuremath{\mathbb{Z}}}
\newcommand{\C}{\ensuremath{\mathbb{C}}}
\newcommand{\F}{\ensuremath{\mathbb{F}}}
\newcommand{\bL}{\ensuremath{\mathbb{L}}}
\newcommand{\cF}{\ensuremath{\mathcal{F}}}
\newcommand{\Kh}{\ensuremath{\mbox{Kh}}}
\newcommand{\SKh}{\ensuremath{\mbox{SKh}}}
\newcommand{\CKh}{\ensuremath{\mbox{CKh}}}
\newcommand\HFL{\widehat{\mbox{HFL}}}
\newcommand\HFK{\widehat{\mbox{HFK}}}
\newcommand\CFL{\widehat{\mbox{CFL}}}

\newcommand{\Id}{\ensuremath{\mbox{\textbb{1}}}}
\begin{document}
\bibliographystyle{plain}

\begin{abstract} We investigate two ``categorified" braid conjugacy class invariants, one coming from Khovanov homology and the other from Heegaard Floer homology. We prove that each yields a solution to the word problem but not the conjugacy problem in the braid group.
\end{abstract}
\maketitle

\section{Introduction}
Recall that the $n$-strand braid group $B_n$ admits the presentation

\[B_n = \left\langle \sigma_1, \ldots, \sigma_{n-1}\,\,\vline\,\,\begin{array}{cl} \sigma_i\sigma_j = \sigma_j\sigma_i & \mbox{if $|i-j| \geq 2,$}\\
\sigma_i\sigma_j\sigma_i = \sigma_j\sigma_i\sigma_j & \mbox{if $|i-j| = 1$}\end{array}\right\rangle,\] where $\sigma_i$ corresponds to a positive half twist between the $i$th and $(i+1)$st strands. Given a word $w$ in the generators $\sigma_1,\dots,\sigma_{n-1}$ and their inverses, we will denote by $\sigma(w)$ the corresponding braid in $B_n$.  Also,  we will write $\sigma\sim\sigma'$ if $\sigma$ and $\sigma'$ are conjugate elements of $B_n$. As with any group described in terms of generators and relations, it is natural to look for combinatorial solutions to the \emph{word} and \emph{conjugacy problems} for the braid group:
\begin{enumerate}
	\item  Word problem: Given words $w$, $w'$ as above, is $\sigma(w) = \sigma(w')$?
	\item Conjugacy problem: Given words $w$, $w'$ as above, is $\sigma(w) \sim \sigma(w')$?
\end{enumerate}

The fastest known algorithms for solving Problems (1) and (2)  exploit the Garside structure(s) of the braid group (cf. \cite{MR2179260} for a survey and \cite{KnotInfo} for an implementation). In addition, any faithful representation of $B_n$ for which the images of the generators and the product rule can be described combinatorially -- for example, the Lawrence-Krammer representation \cite{MR1086755, MR1888796, MR1815219} -- provides a solution to Problem (1).  

The present work is an attempt to understand what two popular combinatorial link homology theories coming from representation theory and symplectic geometry --  namely, Khovanov homology and link Floer homology -- can tell us about Problems (1) and (2). Both theories are powerful enough to detect the unknot \cite{MR2023281,KhDetectUnknot}.  Can they detect the trivial braid? Can they distinguish braid  conjugacy classes?

In this  short note, we explain how to extract braid conjugacy class invariants from  both  theories and we prove that each of these invariants provides a  solution to Problem (1) but not (naively) to Problem (2).  The approaches to Problems (1) and (2) described here are, at present,   more computationally involved than the  solutions alluded to at the top. In that sense, our  results are primarily of theoretical interest. Perhaps more tractable solutions to these problems can be obtained along similar lines in the future as faster algorithms are discovered for computing Khovanov and link Floer homology.

Below, we provide a brief description of the ``categorified" braid conjugacy class invariants studied in this note. Both are  invariants of isotopy classes of oriented links in the solid torus complement of a  neighborhood of an oriented unknot $B \subset S^3$. We will think of $B$ as the compactification of the oriented $z$-axis in $\mathbb{R}^3$,
\[B = \{(r,\theta,z)\,\,\vline\,\,r = 0\} \cup \{\infty\} \subset \mathbb{R}^3 \cup \{\infty\} = S^3,\]
and the complement $S^3-N(B)$ as the product \[A \times I= \{(r,\theta,z)\,\,\vline\,\,r \in [1,2], \theta \in [0, 2\pi), z\in [0,1]\}\] of an annulus $A$ with the interval $I=[0,1]$.

Given a braid $\sigma\in B_n$ (read and oriented, by convention, from top to bottom), we will imagine its oriented closure $\widehat\sigma$ as living in $A\times I$ such that  
$\widehat\sigma$ intersects every disk \[D_{t} = \{(r,\theta,z)\,\,\vline\,\,\theta = t\} \cup \{\infty\}\]
positively in $n$ points, where $D_t$ is oriented so that $B=\partial D_t$. Note that $\widehat\sigma\subset A\times I$ is well-defined up to isotopy and, moreover, that conjugate braids give rise to isotopic closures.  Thus, any invariant of isotopy classes of links in $A \times I$ gives rise to a braid conjugacy class invariant.

The central objects  of study in this paper are the sutured annular Khovanov homology $\SKh(\widehat{\sigma}\subset A\times I)$ (cf. Section \ref{sec:SKh}) and the link Floer homology $\HFL(\widehat\sigma \cup B)$ (cf. Section \ref{sec:HFL}). Both are invariants of the conjugacy class of $\sigma$ per the observation above. As mentioned earlier, we will prove that both can be used to distinguish unequal braids (cf. Theorem \ref{thm:Word}) but  that neither  always distinguishes non-conjugate braids (cf. Theorem \ref{thm:Conj} and Corollary \ref{cor:Conj}). These  invariants share other structural features, including a relationship with the Burau representation \[\Psi:B_n\rightarrow GL_n(\mathbb{Z}[T^{\pm 1}])\] (cf. Remarks \ref{rmk:KhovBurau} and \ref{rmk:HFBurau}). In particular, the graded Euler characteristic of $\HFL(\widehat\sigma \cup B)$ is  (more or less)  the characteristic polynomial of $\Psi(\sigma)$. We will show, however, that the  homology  contains more braid  information than does the polynomial in general (cf. Proposition \ref{prop:Burau}). 

It bears mentioning that the standard (i.e. non-sutured) versions of Khovanov homology and knot Floer homology can also be used to detect the trivial braid, thanks to a combination of recent work on the Khovanov/Floer side with older work of Birman-Menasco on braid foliations. Specifically, Batson-Seed \cite{batson-seed} have shown, building on work of Hedden-Ni \cite{HeddenNi}, which, in turn, builds on work of Kronheimer-Mrowka \cite{KhDetectUnknot},  that if $L_n$ is a link with $n$-components, then $\Kh(L_n)\cong \Kh(U_n)$ iff $L_n=U_n$, where $U_n$ is the $n$-component unlink. Analogously, Ni \cite{GT10102808} has shown, building on work of Ozsv{\'a}th-\Szabo \cite{MR2023281},  that $\HFK(L_n) \cong \HFK(U_n)$ iff $L_n = U_n$.

To determine whether $\sigma = \Id \in \mathfrak{B}_n$, it then suffices to compute $\Kh(\widehat{\sigma})$ (resp. $\HFK(\widehat{\sigma})$). If one obtains the same answer as for $U_n$, then the Batson-Seed (resp. Ni) result implies that $L_n = U_n$. A corollary of Birman-Menasco's main result in \cite{MR1030509} -- that an $n$-strand braid whose closure is $U_n$ is conjugate to the trivial braid -- then tells us that $\sigma = \Id$.

There are two primary reasons for focusing on the invariants $\HFL(\widehat\sigma \cup B)$ and $\SKh(\widehat\sigma)$ in this paper. First, they're more intrinsically braid invariants and carry more conjugacy class information than $\HFK(\widehat\sigma)$ and $\Kh(\widehat\sigma)$, which depend only on the isotopy class of $\widehat\sigma$ as a link in $S^3$, and not at all on its embedding in $S^3-N(B)$. Second, and perhaps most significantly, our proof that $\SKh(\widehat\sigma)$ detects the trivial braid is \emph{entirely combinatorial}. It does not rely on Floer homology or gauge theory at all, unlike the proof that $\Kh(\widehat\sigma)$ detects the trivial braid, which relies on deep results in both Heegaard and instanton Floer homology.

Finally, we note that Proposition \ref{prop:nonRV} and Corollary \ref{cor:PlamRVLV} 
should be of independent interest, as they demonstrate that Plamenevskaya's invariant--while not an effective invariant of transverse links (cf. \cite{Diaz})--does detect interesting geometric features of braids. 

\subsection*{Acknowledgements} The authors wish to thank Christian Blanchet, who asked the question that initiated this investigation, and Stephan Wehrli, for a number of interesting conversations. We owe a special debt of gratitude to Matt Hedden and Liam Watson, for sharing their paper \cite{HeddenBotany}. In it, they show (using a simple but powerful argument that arose during conversations with Jeremy Van Horn-Morris) that if $K$ is a Floer-simple fibered knot in a closed, oriented $3$--manifold $Y$, then the monodromy of $K$ must be trivial. Our proofs of Theorem \ref{thm:Word}.(a) and \ref{thm:Word}.(b) were inspired by this idea. Finally, we thank Eddy Godelle for pointing us to the main result of \cite{gonzmen}.

\section{Categorified Braid Conjugacy Class Invariants}
In this section, we briefly recall the construction of sutured annular Khovanov homology and some basic features of link Floer homology. We will assume the reader is familiar with ordinary Khovanov homology and knot Floer homology. All chain complexes and homology theories considered in this paper are  with coefficients in $\F := \Z/2\Z$. 

\subsection{Sutured Annular Khovanov Homology} \label{sec:SKh}
Sutured annular Khovanov homology was originally defined in \cite{MR2113902} as a categorification of the Kauffman bracket skein module of $A\times I$. It was studied further in \cite{GT07060741,AnnularLinks}, where a connection with sutured Floer homology was discovered (hence, the name). The theory associates to an oriented link $\mathbb{L} \subset A \times I$ a triply-graded vector space \[\SKh(\mathbb{L})=\bigoplus_{i,j,k} \SKh^i(\mathbb{L};j,k),\] which is an invariant of the oriented isotopy class of $\mathbb{L}\subset A\times I$.

Its construction starts with a  projection of $\mathbb{L}$ onto the annulus $A\times \{1/2\}$. This projection   may  be viewed as a planar diagram $D$ in $S^2-\{X,O\}$, where $X$ and $O$ are basepoints in $S^2$ corresponding to the inner and outer boundary circles of  $A\times\{1/2\}$. Forgetting the data of the $X$ basepoint temporarily, we may think of $D$ as a planar diagram in $\mathbb{R}^2 = S^2 - \{O\}$ and form the ordinary bigraded Khovanov complex \[\CKh(D)=\bigoplus_{i,j} \CKh^i(D;j)\] from a cube of resolutions of $D$ in the usual way. Here, $i$ and $j$ are the homological and quantum gradings, respectively.\footnote{The grading, $i$, is really a {\em cohomological} grading, as the Khovanov differential {\em increases} it by 1.} The basepoint $X$ gives rise to a filtration on $\CKh(D)$, and $\SKh(\bL)$ is defined to be the (co)homology of the associated graded object. 

To define this filtration, we  choose an oriented arc from $X$ to $O$ missing all crossings of the diagram $D$. As described in \cite[Sec. 4.2]{JacoFest}, the generators of $\CKh(D)$ are in one-to-one correspondence with {\em enhanced} (i.e., {\em oriented}) resolutions. We define the ``$k$" grading of an oriented resolution to be the algebraic intersection number of this resolution with our oriented arc, up to some overall shift. Roberts proves (\cite[Lem. 1]{GT07060741}) that the Khovanov differential does not increase this extra grading. One therefore obtains a filtration,
\[ 0 \subseteq \ldots \subseteq \cF_{n-1}(D) \subseteq \cF_{n}(D) \subseteq \cF_{n+1}(D) \subseteq \ldots \subseteq \CKh(D),\] where $\mathcal{F}_n(D)$ is the subcomplex of $\CKh(D)$ generated by oriented resolutions with $k$ grading at most $n$. Let \[\mathcal{F}_n(D;j) = \mathcal{F}_n(D)\,\cap \,\bigoplus_i \CKh^i(D;j).\] The sutured annular Khovanov homology groups of $\bL$ are  defined to be  \[\SKh^i(\bL;j,k) := H^i\left(\frac{\mathcal{F}_{k}(D;j)}{\mathcal{F}_{k-1}(D;j)}\right).\] The lemma below follows directly.
\begin{lemma}
\label{lem:sskh}
There is a spectral sequence whose $E_1$ term is $\SKh(\bL\subset A\times I)$ and whose $E_{\infty}$ term (ignoring the $k$ grading) is isomorphic to $\Kh(\bL\subset S^3)$. Moreover, the $d_n$ differential shifts the $(i,j,k)$ multi-grading  by $(1,0,-n)$.\end{lemma}

Given a braid $\sigma$, we will be interested in  $\SKh(\widehat\sigma)$, where $\widehat\sigma\subset A\times I$ is the oriented closure of $\sigma$  as described in the Introduction. 

\begin{remark} \label{rmk:KhovBurau}
It is shown in \cite{HochHom} that the summand of $\SKh(\widehat{\sigma})$ in the next-to-top $k$ grading is equal to the Hochschild homology of the braid bimodule constructed by Khovanov and Seidel in \cite{MR1862802}. This  bimodule detects the trivial braid \cite[Cor. 1.2]{MR1862802} (and its action categorifies the Burau representation \cite[Prop. 2.8]{MR1862802}). These facts together with Theorem \ref{thm:Conj} imply that  a categorical group representation may be faithful even while its Hochschild homology is not  faithful on conjugacy classes. Similar results have been obtained in the Heegaard Floer setting by Hedden and Watson \cite{HeddenBotany} in combination with Lipshitz, Ozsv{\'a}th and Thurston \cite[Thm. 14]{GT10030598}, \cite{FaithfulMCG}. 

On the other hand,  Theorem \ref{thm:Word}  suggests the following  question.

\begin{question} Suppose one has a faithful weak action, in the sense of \cite[Def. 2.6]{MR1862802}, of a group $G$ on a (derived) category of modules over an ($A_\infty$) algebra $A$, where the functor associated to an element $g \in G$ is given by taking a (derived) tensor product with a (derived equivalence class of) bimodule $\mathcal{M}_g$. Let $HH(A, \mathcal{M}_g)$ denote the Hochschild homology of $A$ with coefficients in $\mathcal{M}_g$. Does \[HH(A,\mathcal{M}_g) = HH(A,\mathcal{M}_{\Id})\] necessarily imply that $g = \Id$? 
\end{question}
\end{remark}

\subsection{Link Floer Homology} \label{sec:HFL}
Link Floer homology was defined in \cite{GT0512286} as a  generalization of knot Floer homology and a categorification of the multi-variable Alexander polynomial. One version of the theory associates to an oriented link $\bL\subset S^3$ expressed as a union of $k$ sublinks, $\bL = L_1\cup\cdots \cup L_k$, a graded vector space \[\HFL(\bL) = \bigoplus_{d,A_{L_1},\dots,A_{L_k}}\HFL_d(\bL;A_{L_1},\dots,A_{L_k})\] which is an invariant of the oriented isotopy class of $\bL\subset S^3$. Here, $d$ is the Maslov grading and the $A_{L_i}$ are the Alexander gradings associated to the sublinks $L_i$.\footnote{Typically, one associates an Alexander grading to each \emph{component} of $\bL$. The theory we describe here is obtained by ``flattening" the multi-grading associated to each sublink $L_i$ into a single grading by summing.} 
When $k=1$, $\HFL(\bL)$ as described here is  isomorphic to the knot Floer homology $\HFK(\bL)$ after an overall shift of the Maslov grading 
 \cite[Thm. 1.1]{osz19}.

 Link Floer homology enjoys many symmetries. To begin with, $\HFL(\bL)$ is supported in Alexander multi-gradings that are symmetric about the origin in $\R^k$, \begin{equation}\label{eqn:symm}\HFL_d(\bL;A_{L_1},\dots,A_{L_k})\cong\HFL_{d-2S}(\bL;-A_{L_1},\dots,-A_{L_k}),\end{equation} where $S=\sum_i A_{L_i}$.
Moreover, if $\bL'$ is the link obtained by replacing $L_i$ with $-L_i$, then \begin{equation}\label{eqn:symm2}\HFL_d(\bL;A_{L_1},\dots,A_{L_i},\dots,A_{L_k})\cong\HFL_{d-2A_{L_i}+l_i}(\bL';A_{L_1},\dots,-A_{L_i},\dots,A_{L_k}),\end{equation} where $l_i=\mbox{lk}(L_i,\bL-L_i)$. Similarly, if $m(\bL)$ denote the mirror of $\bL$, then \begin{equation}\label{eqn:symm3}\HFL_d(\bL;A_{L_1},\dots, A_{L_k})\cong \HFL_{2S-d+1-|\bL|}(m(\bL); A_{L_1},\dots, A_{L_k}).\end{equation} 
See \cite[Sec. 5]{most} and \cite[Sec. 8]{osz19} for discussions of these symmetries.

The link Floer homology of $\bL$ is the homology of a  chain complex defined in terms of a multi-pointed Heegaard diagram for $\bL$. In an abuse of notation, we will denote this complex by $\CFL(\bL)$. For each sublink $L_i$, $\CFL(\bL)$ can be realized as  the associated graded object of a filtration on some complex $C(\bL-L_i)$, where  \[H_*(C(\bL-L_i))\cong \HFL(\bL-L_i)\otimes V^{\otimes |L_i|}\]  up to a shift of the Alexander multi-grading, for $k\geq 2$. Here,  $V$ is the triply-graded vector space $\F\oplus\F$ whose summands are supported in Maslov gradings $0$ and $-1$ and Alexander multi-grading $(0,\dots,0)\subset \R^{k-1}$. This leads to the following lemma (the claim about grading shifts follows immediately from the discussions in Subsections 3.7 and 8.1 of \cite{osz19}).


\begin{lemma}
\label{lem:ss}
For $k\geq 2$, there is a spectral sequence whose $E_1$ term is $\HFL(\bL)$ and whose $E_{\infty}$ term (ignoring the $A_{L_i}$ grading) is isomorphic to the vector space obtained from $\HFL(\bL-L_i)\otimes  V^{\otimes |L_i|}$ by shifting each Alexander grading $A_{L_j}$ by $\frac{1}{2}\mbox{lk}(L_j,L_i)$. 
\end{lemma}



The following describes a similar spectral sequence in the case that $k=1$ (cf. \cite{osz19}).

\begin{lemma}
\label{lem:ss2}
For $k=1$, there is a spectral sequence whose $E_1$ term is $\HFL(\bL)$ and whose $E_{\infty}$ term is a rank $2^{|\bL|-1}$ vector space.
\end{lemma}

We  now restrict our attention to the case $\bL = \widehat\sigma\cup B$. Note that $\HFL(\widehat\sigma \cup B)$ is an invariant of the oriented isotopy class of $\widehat\sigma\subset A\times I$.  

\begin{remark}
That $\HFL(\widehat\sigma \cup B)$ is the homology of the associated graded object of a filtration on a complex which computes $\HFL(\widehat\sigma)\otimes V$ (up to an Alexander grading shift) closely parallels the relationship between  $\SKh(\widehat\sigma\subset A\times I)$ and $\Kh(\widehat\sigma\subset S^3)$ described in the previous subsection.
\end{remark}

\begin{remark} \label{rmk:HFBurau} The fact that link Floer homology categorifies the multi-variable Alexander polynomial, combined with an older result of Morton \cite{MR1701681}, implies that $\widehat{HFL}(\widehat{\sigma} \cup B)$ categorifies the characteristic polynomial $\mbox{{det}}(\lambda - \Psi(\sigma))$, where \[\Psi:B_n\rightarrow GL_n(\Z[T^{\pm 1}])\] is the Burau representation of $B_n$. More precisely, 
\[\sum_{d, A_{\widehat\sigma}, A_B} (-1)^d \cdot {\rm rk}\HFL_d(\widehat\sigma \cup B;A_{\widehat\sigma}, A_B)\cdot T^{A_{\widehat\sigma}}\cdot \lambda^{A_B} = \mbox{det}(\lambda - \Psi(\sigma))\cdot (T-1)^{|\widehat\sigma|},\] up to an overall factor of $\pm T^{m_1}\cdot \lambda^{m_2}$ for some $m_1,m_2 \in \frac{1}{2}\Z$. 

On the other hand, since the Burau representation  is not faithful for $n\geq 5$ (cf. \cite{Moody, LongPaton, BigelowBurau5}), Theorem \ref{thm:Word}.(b)  implies that $\HFL(\widehat\sigma\cup B)$ contains strictly more information about $\sigma$ than does $\mbox{{det}}(\lambda - \Psi(\sigma))$:

\begin{proposition} \label{prop:Burau}
	For each $n \geq 5$ there exists a braid  $\sigma \in B_n$ for which \[\mbox{\em det}(\lambda - \Psi(\sigma)) = \mbox{\em det}(\lambda - \Psi(\Id))\,\,\,\mbox{ but }\,\,\,\HFL(\widehat\sigma\cup B) \not\cong \HFL(\widehat\Id\cup B).\]
\end{proposition}
\end{remark}


\section{Trivial Braid Detection: Solution To  Word Problem}

The goal of this section is to prove that $\SKh(\widehat{\sigma})$ and $\HFL(\widehat\sigma\cup B)$ detect the trivial braid and therefore provide solutions to Problem (1), per the theorem below. 

\begin{theorem} \label{thm:Word} Suppose $\sigma\in B_n$.
\begin{enumerate}[(a)]
	\item If $\SKh(\widehat{\sigma}) \cong \SKh(\widehat{\Id})$, then $\sigma =\Id$.
	\item If $\HFL(\widehat\sigma\cup B) \cong \HFL(\widehat\Id\cup B)$ and $\sigma$ is a pure braid, then $\sigma = \Id$.\footnote{If $\sigma$ is not a pure braid, which is easy to check by hand, then $\sigma\neq \Id$.}
	\end{enumerate}
	In particular, both $\SKh(\widehat{\sigma})$ and $\HFL(\widehat\sigma\cup B)$ can be used to give solutions to the word problem in the braid group.
	
\end{theorem}

The proof of Theorem \ref{thm:Word}.(a) relies on certain properties of Plamenevskaya's invariant \cite{Plam} of transverse links. We pause to discuss these first, delaying the proof of Theorem \ref{thm:Word} to Subsection \ref{sec:ProofWord}.

\subsection{Plamenevskaya's Invariant and Trivial Braid Detection} \label{sec:PlamInv}

Let $D_n$ denote the standard unit disk with $n$ distinct marked points $p_1, \ldots, p_n$ positioned along the real axis.

\begin{definition} We say an arc $\gamma: [0,1] \rightarrow D_n$ is {\em admissible} if it satisfies
\begin{enumerate}
	\item $\gamma$ is a smooth imbedding transverse to $\partial D_n$,
	\item $\gamma(0) = -1 \in \C$ and $\gamma(1) \in \{p_1, \ldots, p_n\}$,
	\item $\gamma(t) \in D_n \setminus (\partial D_n \cup \{p_1, \ldots, p_n\})$ for all $t \in (0,1)$, and
	\item $\frac{d\gamma}{dt} \neq 0$ for all $t \in [0,1]$.
\end{enumerate}
\end{definition}

Note that the above definition differs from the one given in \cite[Sec. 4]{baldvv}, where admissible arcs originate and terminate on $\partial D_n$. We will often abuse notation and use $\gamma$ to refer to the image of $\gamma$ in $D_n$.

\begin{definition} Two admissible arcs $\gamma, \gamma'$ are said to be {\em pulled tight} if they satisfy one of:
\begin{itemize}
	\item $\gamma = \gamma'$, or
	\item $\gamma$ and $\gamma'$ intersect transversely, and if $t_1, t_2, t_1', t_2' \in [0,1]$ satisfy the property that $\gamma([t_1, t_2]) \cup  \gamma'([t_1',t_2'])$ bounds an imbedded disk $A \subset D_n$, then $A \cap \{p_1, \ldots, p_n\} \neq \emptyset$ (i.e., $\gamma$ and $\gamma'$ are transverse and form no empty bigons).
\end{itemize}
\end{definition}

Note that if $\gamma, \gamma'$ are admissible arcs, there exist admissible arcs $\delta, \delta'$ isotopic to $\gamma,\gamma'$, resp., such that $\delta, \delta'$ are pulled tight (cf. \cite[Sec. 6]{BraidLeftOrder}).

\begin{definition}
Let $\gamma, \gamma'$ be two admissible arcs in $D_n$. We say $\gamma$ is {\em right} of $\gamma'$ if, when pulled tight via isotopy, the orientation induced by the tangent vectors $\frac{d\gamma}{dt}\vline_{t=0}, \frac{d\gamma'}{dt}\vline_{t=0}$ agrees with the standard orientation on $D \subset \C$.
\end{definition}

There is a well-known isomorphism $\mathfrak{B}_n \cong \mathcal{MCG}(D_n)$ obtained by identifying the elementary Artin generator $\sigma_i$ (resp., $\sigma_i^{-1}$) with a homeomorphism of $D_n$ that acts as the identity outside of a small disk enclosing $p_i, p_{i+1}$ and acts as a $180^\circ$ CCW (resp., CW) rotation on a (slightly smaller) disk enclosing $p_i, p_{i+1}$. Note that when a braid is viewed as a mapping class, it acts on $D_n$ {\em on the right} (since Artin braids are read from left to right).

Via the above isomorphism, $\mathfrak{B}_n$ acts on the set of isotopy classes of admissible arcs. Let $(\gamma)\sigma$ denote the image of (the isotopy class of) $\gamma$ under $\sigma \in \mathfrak{B}_n$.

\begin{remark} \label{rmk:LRConvention} Note that in \cite{BraidLeftOrder} (cf. Prop. 1.1.3), positive (resp., negative) Artin generators are identified with CW (resp., CCW) rotations so that braids may act on $D_n$ {\em on the left}. We chose our convention to match those in \cite{MR2318562} and \cite{baldvv}.
\end{remark}
 
\begin{definition}
Let $\sigma \in \mathfrak{B}_n$. We say $\sigma$ is {\em right-veering} if, for all admissible arcs $\gamma$, $(\gamma)\sigma$ is right of $\gamma$ when pulled tight.
\end{definition}

\begin{remark}
We may analogously define the notion of {\em left-veering}. A braid $\sigma \in \mathfrak{B}_n$ is left-veering (resp. right-veering) iff its mirror $m(\sigma) \in \mathfrak{B}_n$ is right-veering (resp. left-veering).
\end{remark}

The following lemma is well-known:

\begin{lemma} \label{lem:RVLVId} If $\sigma \in \mathfrak{B}_n$ is both right-veering and left-veering, then $\sigma = \Id$.
\end{lemma}

\begin{proof} If $\sigma \in \mathfrak{B}_n$ is both right- and left-veering, it must send each admissible arc to an isotopic admissible arc. Consider the collection $\{Q_1, \ldots, Q_n\}$ of admissible arcs pictured in Figure \ref{fig:Qbasis}. A straightforward inductive argument then shows that 
$\sigma$ is isotopic to a map which fixes each $Q_i$ in this collection. The Alexander lemma (cf. \cite[Lem. 2.1]{MR2850125}) then implies that $\sigma = \Id$, as desired.
\end{proof}

\begin{figure}
\begin{center}
\resizebox{1.5in}{!}{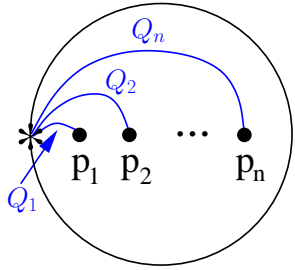}
\end{center}
\caption{A collection of admissible arcs $Q_1, \ldots, Q_n \subset D_n$ whose complement is a disk}
\label{fig:Qbasis}
\end{figure}

Let $\psi(\widehat{\sigma}) \in \Kh(\widehat{\sigma})$ denote Plamenevskaya's invariant \cite{Plam} of the transverse link in the tight contact structure on $S^3$ represented by the braid $\widehat{\sigma}$. 

\begin{proposition} \label{prop:nonRV} If $\sigma$ is not right-veering, then $\psi(\widehat{\sigma}) = 0$.
\end{proposition}

\begin{proof} Suppose $\sigma$ is not right-veering. Then there exists some admissible arc $\gamma$ which is sent to the left by $\sigma$. Let $\tau \in \mathfrak{B}_n$ satisfy $\gamma = (Q_1)\tau$. Then $(Q_1)\tau\sigma\tau^{-1}$ is left of $Q_1$. By \cite[Prop. 6.2.7]{BraidLeftOrder}, 
there exists a {\em $\sigma_i$--negative} word (\cite[Defn. 1.2.3]{BraidLeftOrder}) representing $\tau\sigma\tau^{-1}$ for some $i \in \{1, \ldots, n\}$. Such a word contains at least one letter $\sigma_i^{-1}$ and no letters $\sigma_i$.  It then follows from  \cite[Proposition 3]{Plam} that $\psi(\widehat{\tau\sigma\tau^{-1}}) = \psi(\widehat\sigma) = 0$.

\end{proof}

\begin{corollary} \label{cor:PlamRVLV} If $\psi(\widehat{\sigma}) \neq 0$ and $\psi(m(\widehat{\sigma})) \neq 0$, then $\sigma = \Id$.
\end{corollary}

\begin{proof} By Proposition \ref{prop:nonRV}, $\sigma$ is both right- and left-veering. By Lemma \ref{lem:RVLVId}, $\sigma = \Id$.
\end{proof}


\subsection{Proof of Theorem \ref{thm:Word}} \label{sec:ProofWord}

\begin{proof}[Proof of Theorem \ref{thm:Word}.(a)]

Suppose $\SKh(\widehat{\sigma}) \cong \SKh(\widehat{\Id})$. According to Roberts \cite{GT07060741}, the invariant $\psi(\widehat\sigma)$ is the image of the generator of the bottommost $k$ grading of $\SKh(\widehat\sigma)$ under the spectral sequence from $\SKh(\widehat\sigma)$ to $\Kh(\widehat\sigma)$. By direct computation, $\SKh(\widehat{\sigma})$ is supported in a single homological grading, which implies that the spectral sequence from $\SKh(\widehat{\sigma})$ to $\Kh(\widehat{\sigma})$ collapses immediately. It follows that $\psi(\widehat{\sigma}) \neq 0$. The symmetry of sutured Khovanov homology under taking mirrors \cite[Lem. 2]{GT07060741} implies that $\SKh(m(\widehat{\sigma}))$ is also supported in a single homological grading. It follows that  $\psi(m(\widehat{\sigma})) \neq 0$ as well. Corollary \ref{cor:PlamRVLV} then tells us that $\sigma = \Id$, as desired.
\end{proof}

The proof of Theorem \ref{thm:Word}.(b) is very similar in spirit to that of Theorem \ref{thm:Word}.(a). The analogue of Proposition \ref{prop:nonRV} was proved in \cite{baldvv}. 

\begin{proof}[Proof of Theorem \ref{thm:Word}.(b)]

For  $\sigma\in B_n$, Baldwin, Vela-Vick and V{\'e}rtesi define in \cite{baldvv} a class $\widehat t(\widehat\sigma)\in \HFL(m(\widehat\sigma))$ which is an invariant of the transverse link represented by  $\widehat\sigma$ (and agrees, for transverse \emph{knots}, with the invariants defined in \cite{oszt,lossz}). We  show below that if $\sigma$ is a pure braid and $\HFL(\widehat\sigma\cup B) \cong \HFL(\widehat\Id\cup B)$, then  $\widehat t (\widehat\sigma)\neq 0$ and $\widehat t (\widehat{\sigma^{-1}})\neq 0$. An analogue of Corollary \ref{cor:PlamRVLV} then implies that $\sigma$ is equal to $\Id$. 

Suppose $\sigma$ is a pure braid and $\HFL(\widehat\sigma\cup B) \cong \HFL(\widehat\Id\cup B)$. We first compute the latter. Let us denote $\Id$ by $\Id_n$ to indicate that it is the trivial braid on $n$-strands. Note that $\widehat\Id_n\cup B$ is isotopic to link gotten by taking the connected sum of the positive Hopf link $\widehat\Id_1\cup B$ with itself $n$ times along the component $B$. As computed in \cite[Sec. 12]{osz19} (cf. also \cite[Sec. 4]{CombHFK}), \[\HFL(\widehat\Id_1\cup B) \cong (V_1\otimes V_2)[1/2,1/2].\] Here, $V_i$ is the triply-graded vector space $\F\oplus \F$ whose first and second summands are supported in Maslov gradings $0$ and $-1$ and Alexander bi-gradings $(0,0)$ and $-e_i$, where $e_i$ is the $i$th standard basis vector in $\R^2$. The $[1/2,1/2]$ indicates that we have shifted the Alexander bi-grading by $(1/2,1/2)$. Our initial assumption and the K{\"u}nneth formula in \cite{osz19}  imply that \begin{equation}\label{eqn:hfl}\HFL(\widehat\sigma\cup B) \cong \HFL(\widehat\Id_n\cup B)\cong \HFL(\widehat\Id_1\cup B)^{\otimes n}\cong (V_1^{\otimes n}\otimes V_2^{\otimes n})[n/2,n/2].\end{equation}

To compute  $\widehat t(\widehat\sigma)$, we need to know the link Floer homology of $m(\widehat\sigma\cup -B)$. From (\ref{eqn:hfl})  together with the formulae in (\ref{eqn:symm2}) and (\ref{eqn:symm3}), it follows   that  \[\HFL(m(\widehat\sigma\cup -B))\cong (V_1^{\otimes n}\otimes V_2^{\otimes n})[n/2,n/2]\] as well. 
 Lemma \ref{lem:ss} then implies that there is a spectral sequence with $E_1$ term \[\HFL(m(\widehat\sigma\cup -B))[-n/2,0] \cong (V_1^{\otimes n}\otimes V_2^{\otimes n})[0,n/2]\] and whose $E_{\infty}$ term (ignoring the $A_{m(-B)}$ grading) is isomorphic to $\HFL(m(\widehat\sigma))\otimes V$, where $V$ is the bi-graded vector space $\F\oplus \F$ whose   summands are supported in Maslov gradings $0$ and $-1$ and $A_{m(\widehat\sigma)}$ grading $0$. To reduce clutter, we introduce the notation 
\begin{align*}
H(\widehat\sigma,B) & := \HFL(m(\widehat\sigma\cup -B))[-n/2,0],\\
H(\widehat\sigma) &:=\HFL(m(\widehat\sigma))\otimes V.
\end{align*} 
Note that the top Maslov graded piece $H_{top}$ (in Maslov grading $-n$) of the portion of  $H(\widehat\sigma,B) $ in the bottommost $A_{m(-B)}$ grading $-n/2$ has rank one. The element  $\widehat t(\widehat\sigma)$ can be characterized in terms of the image in $H(\widehat\sigma)$ of the generator of $H_{top}$ under the spectral sequence from $H(\widehat\sigma,B) $ to $H(\widehat\sigma) $. In particular, $\widehat t(\widehat\sigma)$ is non-zero iff $H_{top}$ survives under this spectral sequence \cite[Corollary 6.9]{baldvv}.
 
Note that $H(\widehat\sigma,B) $ is supported in $A_{m(\widehat\sigma)}$ gradings $-n\leq A_{m(\widehat\sigma)}\leq 0$. Since $H(\widehat\sigma) $ is  supported in $A_{m(\widehat\sigma)}$ gradings that are symmetric about the origin (as discussed in Subsection \ref{sec:HFL}) it must be that the portion of $H(\widehat\sigma,B) $ in negative $A_{m(\widehat\sigma)}$ gradings dies in the spectral sequence. Moreover, the portion  in $A_{m(\widehat\sigma)}$ grading $0$ has rank $2^n$. Therefore, \[\mbox{rk}(H(\widehat\sigma) )\leq 2^n.\] Now, Lemma \ref{lem:ss2}  implies that there is a spectral sequence whose $E_1$ term is $H(\widehat\sigma) $ and whose $E_{\infty}$ term is a rank $2^{|m(\widehat\sigma)|}$ vector space. Since $\sigma$ is a pure braid, $|m(\widehat\sigma)| = n$. Thus,  \[\mbox{rk}(H(\widehat\sigma))\geq 2^n.\] It  follows that $\mbox{rk}(H(\widehat\sigma) )=2^n.$ It must therefore be the case that the entire portion of $H(\widehat\sigma,B)$ in $A_{m(\widehat\sigma)}$ grading $0$ survives in the spectral sequence from $H(\widehat\sigma,B)$ to $H(\widehat\sigma)$. Since $H_{top}$ is contained in this portion we may deduce that $\widehat t(\widehat\sigma)$ is non-zero. Hence, $\sigma$ is right-veering by \cite[Theorem 1.4]{baldvv}.
 
Note that $\widehat{\sigma^{-1}}\cup B$ is isotopic to $m(-\widehat\sigma \cup B)$. The formulae in (\ref{eqn:symm2}) and (\ref{eqn:symm3}) therefore imply that \[\HFL(\widehat{\sigma^{-1}}\cup B) \cong \HFL(\widehat\sigma\cup B)\cong \HFL(\widehat\Id\cup B)\] as well. It follows that $\widehat t(\widehat{\sigma^{-1}})$ is also non-zero. As before, this implies that $\sigma^{-1}$ is right-veering and, hence, that $\sigma$ is left-veering. Since $\sigma$ is both right- and left-veering, it is equal to $\Id$.
\end{proof}

The invariants $\SKh(\widehat\sigma)$ and $\HFL(\widehat\sigma\cup B)$ can be used to solve the word problem in $B_n$ as follows. Suppose $w$ and $w'$ are words in the generators $\sigma_1,\dots,\sigma_n$ and their inverses  representing  the braids $\sigma(w)$ and $\sigma(w')$ in $B_n$. Let  $\sigma = \sigma(w)\cdot(\sigma(w'))^{-1}$ and note that \begin{equation}\label{eqn:equiv}\sigma = \Id \Longleftrightarrow \sigma(w) = \sigma(w').\end{equation} In particular, if $\sigma$ is not a pure braid, then $\sigma(w) \neq \sigma(w').$ If $\sigma$ is a pure braid, compute $\SKh(\widehat\sigma)$ or $\HFL(\widehat\sigma\cup B)$ (which can be done combinatorially) and apply Theorem \ref{thm:Word} with  (\ref{eqn:equiv}) in mind.


\section{Transverse Mirror Invariance: No Solution To Conjugacy Problem}
In this section, we show that $\SKh(\widehat\sigma)$ and $\HFL(\widehat\sigma\cup B)$ cannot always distinguish non-conjugate braids. This fact is made precise in Corollary \ref{cor:Conj} of Theorem \ref{thm:Conj} below.  First, some definitions and remarks.

\begin{figure}[ht]
\labellist
\pinlabel $u$ at 11 151
\pinlabel $v$ at 23 116
\pinlabel $\pm$ at 134 116
\pinlabel $w$ at 11 82

\pinlabel $u$ at 237 151
\pinlabel $\pm$ at 249 116
\pinlabel $v$ at 360 116
\pinlabel $w$ at 237 82
\pinlabel $\mbox{flype}$ at 191 100
\endlabellist
\centering
\includegraphics[width=8.5cm]{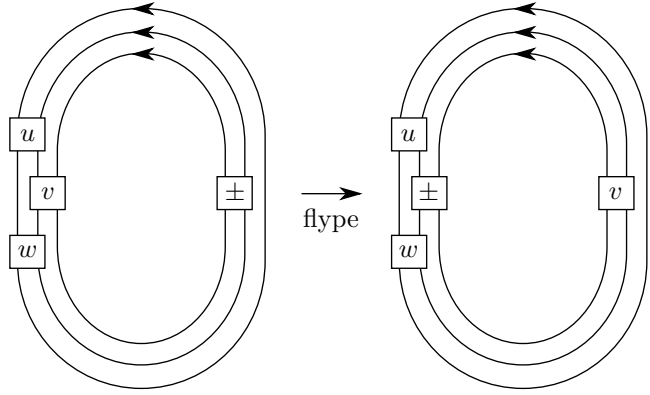}
\caption{Left, a 3-braid closure $\widehat\sigma$. Each of the boxes labeled $u,v,w$ represents some number of positive or negative half twists. The box labeled $\pm$ represents a single positive or negative half twist. Right, the closed braid obtained from $\widehat\sigma$ by a flype. It is oriented isotopic to $\widehat\sigma^r$ in $A\times I$.}
\label{fig:Flype}
\end{figure}

\begin{definition} Let $\sigma \in B_n$ and suppose $\sigma = \sigma(w)$ for some word $w$ in the generators $\sigma_1,\dots,\sigma_n$ and their inverses. Then the {\em reverse} of $\sigma$ is the braid $\sigma^r = \sigma(w^r)$, where $w^r$ is the word obtained from $w$ by reversing the order of its letters. 
\end{definition}

\begin{remark} Thought of as transverse links in the tight contact structure on $S^3$, the closure  $\widehat{\sigma}^r$ is the \emph{transverse mirror} of $\widehat{\sigma}$, as defined by Ng in \cite[Defn. 4.5]{MR2807087}. 
\end{remark}

\begin{theorem} 
\label{thm:Conj} Suppose $\sigma \in B_n$. Then
\begin{enumerate}[(a)]
	\item $\SKh(\widehat{\sigma}) \cong \SKh(\widehat{\sigma}^r)$,
	\item $\HFL(\widehat{\sigma} \cup B) \cong \HFL(\widehat{\sigma}^r \cup B)$.
\end{enumerate}
\end{theorem}

\begin{proof}[Proof of Theorem \ref{thm:Conj}.(a)] Let $D$ and $D^r$ denote the projections of $\widehat\sigma$ and $\widehat\sigma^r$ onto $A\times \{1/2\}$ as described in Subsection \ref{sec:SKh}. There is an obvious bijection between oriented resolutions of $D$ and $D^r$ which preserves the $(i,j,k)$ triple-grading and commutes with the differentials in the complexes $\CKh(D)$ and $\CKh(D^r)$. (The projection $D^r$ is exactly what you see by looking at $D$ from below and reversing orientations.) It follows  that $\SKh(\widehat{\sigma}) \cong \SKh(\widehat{\sigma}^r)$. 
\end{proof}

\begin{proof}[Proof of Theorem \ref{thm:Conj}.(b)]
Note that if $\bL = \widehat\sigma\cup B$, then $\widehat\sigma^r\cup B$ is oriented isotopic to $-\bL$ in $S^3$. It follows from (\ref{eqn:symm}) and (\ref{eqn:symm2}) that $\HFL(\bL) \cong \HFL(-\bL)$ for any oriented link $\bL$.
\end{proof}

\begin{corollary} \label{cor:Conj}
There exist infinitely many pairs $(\sigma,\sigma')\in B_3\times B_3$ such that $\sigma\not\sim\sigma'$ but \[\SKh(\widehat{\sigma}) \cong \SKh(\widehat{\sigma}') \,\,\,\mbox{ and }\,\,\,\HFL(\widehat{\sigma} \cup B) \cong \HFL(\widehat{\sigma}' \cup B).\]
\end{corollary}

\begin{proof} Suppose $\sigma$ and $\sigma'$ are 3-braids for which  $\widehat\sigma$ and $\widehat\sigma'$ are related by a \emph{flype}, as described in \cite[Fig. 1.2]{MR2468377} and illustrated in  Figure \ref{fig:Flype}. Then $\widehat\sigma'$ is clearly oriented isotopic to $\widehat{\sigma}^r$ in $A\times I$. In particular, $\widehat\sigma'$ and $\widehat\sigma$ are transverse mirrors, a fact  first  observed by Ng in \cite{MR2807087}. Theorem \ref{thm:Conj} then implies that \[\SKh(\widehat{\sigma}) \cong \SKh(\widehat{\sigma}') \,\,\,\mbox{ and }\,\,\,\HFL(\widehat{\sigma} \cup B) \cong \HFL(\widehat{\sigma}' \cup B).\] On the other hand, Birman and Menasco have enumerated infinitely many conjugacy classes of $3$-braids admitting {\em non-degenerate} flypes; i.e. flypes which do not preserve conjugacy class  \cite[Tab. 2]{MR2468377}.
\end{proof}

\begin{remark} Let $H(\widehat\sigma)$ denote either of the theories $\SKh(\widehat\sigma)$ or $\HFL(\widehat\sigma\cup B)$. Given Theorem \ref{thm:Conj} and Corollary \ref{cor:Conj}, one wonders, for braids $\sigma$, $\tau$ in $B_n$, whether $H(\widehat\sigma)\cong H(\widehat\tau)$ implies that $\sigma \sim \tau$ or $\sigma\sim\tau^r$. This is almost certainly too optimistic. It is more reasonable to ask the following. 
\begin{question}
\label{ques:power}
Does $H(\widehat\sigma^k)\cong H(\widehat\tau^k)$ for all integers $k\geq 0$ imply that $\sigma \sim \tau$ or $\sigma\sim\tau^r$? What if we also assume that $\sigma$ and $\tau$ are alternating braids?
\end{question}

This question is motivated to some extent by work of Gonz{\'a}les-Meneses, who proves in \cite{gonzmen} that if  $\sigma^k = \tau^k$ for some $k$ then $\sigma\sim\tau$. The notion that the answer to Question \ref{ques:power} might be ``yes" for alternating braids is inspired by Greene's recent result that if $D_1$, $D_2$ are reduced, alternating planar diagrams for knots $K_1$, $K_2$ in $S^3$ and $\widehat{\mbox{HF}}(\Sigma(S^3,K_1))\cong \widehat{\mbox{HF}}(\Sigma(S^3,K_2))$, then $D_1$, $D_2$ are related by a sequence of Conway mutations \cite{greene}.
\end{remark}

\bibliography{HFL_BraidConj}
\end{document}

%% file: Figures/Qbasis.pdf_t
\begin{picture}(0,0)%
\epsfig{file=Qbasis.pdf}%
\end{picture}%
\setlength{\unitlength}{3158sp}%
\begingroup\makeatletter\ifx\SetFigFont\undefined%
\gdef\SetFigFont#1#2#3#4#5{%
  \reset@font\fontsize{#1}{#2pt}%
  \fontfamily{#3}\fontseries{#4}\fontshape{#5}%
  \selectfont}%
\fi\endgroup%
\begin{picture}(5233,4824)(3166,-5618)
\put(3166,-4531){\makebox(0,0)[lb]{\smash{{\SetFigFont{25}{30.0}{\rmdefault}{\mddefault}{\updefault}{\color[rgb]{0,0,1}$Q_1$}%
}}}}
\put(4966,-2371){\makebox(0,0)[lb]{\smash{{\SetFigFont{25}{30.0}{\rmdefault}{\mddefault}{\updefault}{\color[rgb]{0,0,1}$Q_2$}%
}}}}
\put(5446,-1441){\makebox(0,0)[lb]{\smash{{\SetFigFont{25}{30.0}{\rmdefault}{\mddefault}{\updefault}{\color[rgb]{0,0,1}$Q_n$}%
}}}}
\end{picture}%